\documentclass[a4paper, 11pt,reqno]{amsart}
\title[Neumann problem in the Heisenberg Group]{The Neumann problem for Kohn-Laplacian on the Heisenberg Group $\mathbb H_n$}
\usepackage[ansinew]{inputenc}
\usepackage{yfonts}
\usepackage{amscd}
\usepackage{amssymb, latexsym}
\usepackage{mathrsfs}
\usepackage{fancyhdr}
\usepackage{enumerate}
\usepackage{setspace}
\usepackage{latexsym}
\def\H{\mathbb H}
\def\C{\mathbb C}
\def\R{\mathbb R}
\def\N{\mathbb N}
\def\ben{\begin{eqnarray*}}
\def\een{\end{eqnarray*}}
\def\beg{\begin{eqnarray}}
\def\eeg{\end{eqnarray}}
\newtheorem{thm}{Theorem}[section]

\newtheorem{defn}[thm]{Definition}
\newtheorem{lem}[thm]{Lemma}
\author[S. Dubey, A. kumar and M. M. Mishra ]{Shivani Dubey, Ajay Kumar* and Mukund Madhav Mishra}
\address{Department of Mathematics, University of Delhi, Delhi, India}
\email{shvndb@gmail.com}
\address{Department of Mathematics, University of Delhi, Delhi, India}
\email{akumar@maths.du.ac.in}
\address{Department of Mathematics, Hans Raj College, University of Delhi, Delhi, India}
\email{mukund.math@gmail.com}
\thanks{*Corresponding author. Email: akumar@maths.du.ac.in}
\numberwithin{equation}{section}
\date{}
\begin{document}
\begin{abstract}
Existence and uniqueness of the solution of the Neumann problem for the Kohn-Laplacian on the Kor\'{a}nyi ball of the Heisenberg group $\H_n$ are discussed. Explicit representations of Green's type function (Neumann function) for the Kor\'{a}nyi ball in $\H_n$ for circular functions has been obtained. This function is then used on above region in $\H_n$ to solve the inhomogeneous Neumann boundary value problem for circular data.
\end{abstract}
\maketitle
\section{Introduction}
The boundary value problems have much broader utility in physics, electrostatics and magnetic field etc. These boundary conditions have a physical interpretation where we keep the ends of our rod at freezing without regulating the heat flow in or out of the endpoints.The concept of Green type function for Neumann problem or Neumann function has been considered by several authors in different domains of $\R^n$ \cite{copa, naya, xu}.\\
The classical Neumann problem is to find a function $u:\; \Omega \rightarrow \R ,\; \Omega$ is a bounded domain in $\R^n$ and $\Delta$ is usual Laplacian on $\Omega$  such that $u \in C^2(\Omega) \cap C^1(\bar{\Omega})$ and 
\begin{eqnarray} \label{1}
\Delta u &=& 0\;\; \text{in} \;\; \Omega \nonumber\\
\frac{\partial u}{\partial n} &=& f \;\; \text{on}\;\; \partial \Omega
\end{eqnarray}
where $n$ is the outward normal to $\Omega$ and $f$ is a given continuous function on $\partial \Omega$. Firstly it should be noted that $(\ref{1})$ cannot have a solution for every $f$. This is clear from the physical interpretation of $(\ref{1})$ as a steady-state heat conduction problem. There are no sources in the region $\Omega$ and the heat flow is prescribed on $\partial \Omega$. These conditions are consistent with the steady state only if the total heat flow through $\partial \Omega$ vanishes. Accordingly, a solution of $(\ref{1})$ will exist only if 
\begin{eqnarray} \label{2}
\int_{\partial \Omega}{f(\xi)d\xi}=0.
\end{eqnarray} However, the solution is unique up to a constant whenever it exists.
 A Neumann's function $N(x,y)$ for $\Delta$ on $\Omega$ is a function $N$ having the properties $x \rightarrow N(x,y)$ belongs to $C^2(\bar{\Omega}\setminus \{y\})$,$$\Delta_x N(x,y)=0,\; \text{for} \;x \in \Omega\; \text{and}\; \frac{\partial N(x,y)}{\partial n_x}=0,\; \text{for}\; x \in \partial \Omega.$$ When the existence of the solution is known, the solution, up to a constant is given by $$u(y)=-\int_{\partial\Omega}{N(x,y)f(x)d\sigma_x}.$$ The theory of Neumann function has been widely studied in the case of domains in Euclidean spaces. The Neumann function for the sphere in $\R^3$ has been constructed using the classical method of images and expressed in terms of eigenvalues associated with the surface, leading to an analogue of the Poisson integral as a solution to the Neumann problem for the sphere \cite{naya}.\\ On Euclidean spaces, the Laplace operator is an elliptic operator which is invariant under translations and rotations, and is homogeneous of degree two. Analogously one can define a Laplace operator on non-Euclidean spaces such as Lie groups. The next best class of operators to look at can be one with relaxed regularity conditions such as sub-ellipticity or hypoellipticity \cite{evan, stein}. A Laplacian with this relaxed condition is, in literature, called a sub-Laplacian. The Heisenberg group $\H_n$ is one of the simplest examples of a non commutative Lie group and the Heisenberg sub-Laplacian also known as the Kohn-Laplacian is the most important prototype of the sub-Laplacian. The fundamental solution for the sub-Laplacian on the Heisenberg group was first given by Folland \cite{fol}. The Dirichlet problem on the Heisenberg group and existence of unique solution was discussed in \cite{gav, jeri}. However, the general Green function is not known for any domain in $\H_n, n>1$.\\ Our aim is to formulate a boundary value problem with a Neumann type boundary condition for the Koranyi ball $B$ on the Heisenberg group
\ben
L_0 u&=&f \;\text{in}\; B,\\
\partial ^ \perp u&=&g \;\text{on}\; \partial B,
\een $\partial ^ \perp$ is an operator akin to normal derivative and will be described later. We establish existence and uniqueness of the solution of such a problema and, in certain particular case, we obtain a kernel function $N_B(\eta,\xi)$ to represent the solution by the Poisson like formula 
$$u(\eta)=\int_B{N_B(\eta,\xi)f(\xi)dv(\xi)}-\int_{\partial B}{N_B(\eta,\xi)g(\xi)d\sigma(\xi)}.$$ \\
\section{The Heisenberg group $\H_n$ and Horizontal normal vectors}
\subsection{The Heisenberg Group}
The Heisenberg group $\mathbb{H}_n$ is the set of all pairs $[z,t]\in \mathbb{C}^n\times\mathbb{R}$ with the operation
$$[z,t].[z',t']=[z+z',t+t'+2 \Im (z.\bar{z}')],$$ where $z=(z_1, \ldots, z_n), z.\bar{z}'=z_1\bar{z_1}'+ \ldots +z_n\bar{z_n}'$. 
A basis of left invariant vector fields on $\H_n$ is given by $\{Z_j,\bar{Z}_j,T:1\leq j\leq n\}$ where, 
\begin{eqnarray*}
Z_j &=& \partial_{z_j}+i\bar{z}_j\partial_t ;\\
\bar{Z}_j &=& \partial_{\bar{z}_j}-iz_j\partial_t ;\\
T &=& \partial_t.
\end{eqnarray*}
If we write $z_j=x_j+iy_j$ and define,
\ben
X_j&=& \partial _{x_j}+2y_j \partial _t,\\
Y_j&=& \partial _{y_j}-2x_j \partial _t,
\een
then, $$Z_j=\frac{1}{2}(X_j-iY_j),$$ and $\{X_j, Y_j, T\}$ is a basis. 
The sublaplacian $\textsl{L}$ on $\mathbb{H}_n$ is explicitly given by
$$\textsl{L}=-\sum_{j=1}^n\left(X_j^2 + Y_j^2\right).$$
Let $L_0$ denote the slightly modified subelliptic operator $-\frac{1}{4}\textsl{L}$. The natural gauge on $\mathbb{H}_n$ is given by
$$N([z,t])=({|z|}^4+t^2)^\frac{1}{4}.$$
The fundamental solution for $L_0$ on  $\mathbb{H}_n$ with pole at identity is given in \cite{fol} as 
$$g_e(\xi)=g_e([z,t])=a_0(|z|^4+t^2)^{-\frac{n}{2}},$$
where,
$$a_0=2^{n-2}\frac{(\Gamma(\frac{n}{2}))^2}{\pi^{n+1}},$$ 
is the normalization constant and $\xi=[z,t]$. The fundamental solution of $L_0$ with pole at $\eta$ is given by

$$g_\eta(\xi)=g_e(\eta^{-1}\xi).$$
From \cite{kor6}, for $\eta=[z',t']$ and $\xi=[z,t]$,

$$g_\eta(\xi)=a_0|C(\eta,\xi)-P(\eta,\xi)|^{-n},$$
where,

$C(\eta,\xi)=|z|^2+|z'|^2+i(t-t')$  and $P(\eta,\xi)=2z.\bar{z'}.$\\
For an integrable function $f$ on $\mathbb{H}_n$, we denote the average of $f$ by

$$\bar{f}([z,t])=\frac{1}{2\pi}\int_0^{2\pi}f([e^{i\theta}z,t])d\theta.$$
A function $f$ is said to be circular if $f([z,t])=\bar{f}([z,t])$, for $[z,t] \in \H_n$.\\
As in \cite{kor6}, the average of the fundamental solution with pole at $\eta$ is given by
$$\bar{g}_\eta(\xi)=a_0|C(\eta,\xi)|^{-n}F\left(\frac{n}{2},\frac{n}{2};n;\frac{|P(\eta,\xi)|^2}{|C(\eta,\xi)|^2}\right),$$
$F$ being the Gaussian hypergeometric function \cite{EDR}.\\
\subsection{Horizontal Normal Vectors}
As in \cite{gav2} and \cite{kor83}, we define a singular Riemannian metric $(M_0)$ as follows. On the linear span of $X_j$, $Y_j (1 \leq j \leq n)$ we define an inner product ${\langle .,. \rangle_0}$ by the condition that the vectors $X_j,\; Y_j$ form an orthonormal system. For vectors not in this span we say that they have infinite length. A vector is said to be horizontal if it has finite length. The horizontal gradient $\nabla_0F$ of a function $F$ on $\H_n$ is defined as the unique horizontal vector such that $$\langle \nabla_0F,V \rangle_0=V. F,$$ where $V.F$ is action of $V$ on $F$, for all horizontal vectors $V$. $\nabla_0F$ can be explicitly written as
$$\nabla_0F=\sum_j{\{(X_jF)X_j+(Y_jF)Y_j\}}.$$
 If a hypersurface in $\H_n$ is given as the level set of a smooth function $F$, at any regular point i.e. a point at which $\nabla_0 F\ne 0$ the horizontal normal unit vector will be defined by $$\frac{\partial}{\partial n_0}:=\frac{1}{||\nabla_0F||_0}\nabla_0F.$$ More precisely, this is the horizontal normal pointing outwards for the domain $\{F<0\}$. A point $\zeta$ on the surface is termed ``characteristic" if $\nabla_0 F(\zeta)=0.$ For smooth $F$, the set of characteristic points forms a lower dimensional (and hence of measure zero) subset of the boundary. We define $n_0=\frac{\partial}{\partial n_0}I$, where $I$ is the identity function of any open neighbourhood of the surface $\{F=0\}.$
\section{The Neumann Problem and its well posedness}
In this section, we establish the uniqueness and existence of the Neumann boundary value problem for Kor\'{a}nyi ball in $\H_n$.
\subsection{Formulation of the Neumann problem}
 Let $D$ be a bounded domain with a smooth boundary $\partial D$ given as the level set of a smooth function $G$ i.e, $\partial D=\{\zeta\in \H_n : G(\zeta)=0 \}$. We define the class  $\emph{C}_*(D)$ as follows $$\emph{C}_*(D)=\{f \in C^2(D)\cap C(\bar{D}):\;\lim_{\zeta \rightarrow \zeta_0} \frac{\partial}{\partial n_0}f(\zeta)\;\text{exists for all characteristic points}\; \zeta \in \partial D\}$$ where the limit in above definition is taken under the relative topology of $\partial D$. Now, define $\partial^\perp : \emph{C}_*(D) \rightarrow C(\partial D) $  $$\partial ^ \perp f(\zeta)=\frac{\partial f}{\partial n_0}(\zeta)\; \text{if}\; \zeta\; \text{is non characteristic point of}\; \partial D,$$
$$ \partial^ \perp f(\zeta_0)=\lim_{\zeta \rightarrow \zeta_0}\frac{\partial f}{\partial n_0}(\zeta),\; \text{when}\; \zeta_0\; \text{is a characteristic point of}\; \partial D.$$ Since characteristic points form a lower dimensional subset of the boundary therefore Gaveau's Green formula \cite{gav} can be written as 
\beg \label{***}
\int_D{(uL_0v-vL_0u)dv(\xi)}=\int_{\partial D}{\left(u\partial^\perp v-v\partial^\perp u\right)d\sigma(\xi)},
\eeg where the surface element $d\sigma$ can be given as $$d\sigma=\frac{\left\|\nabla_0 G\right\|_0}{\left\|\nabla G\right\|}ds,$$ $ds$ is Euclidean surface element. 
 The Neumann problem for any domain $D$ in $\H_n$ is to find a function $u \in \emph{C}_*(D)$ such that
\begin{eqnarray} \label{7}
L_0 u&=& 0 \;\text{in}\; D\nonumber\\
\partial ^ \perp u &=& g \;\text{on}\; \partial D 
\end{eqnarray} where $g$ is a continuous function defined on $\partial D$,
is unique upto a constant.
\begin{lem} \label{y}
(Green's first identity) If $D$ is a bounded domain in $\H_n$ with smooth boundary $\partial D$ and $u,\;v$ are $C^1$ functions on $\bar{D}$, then
$$\int_{\partial D}{v\partial^\perp u d\sigma}=\int_D{(vL_0u+\nabla_0 v.\nabla_0 u)d\eta}.$$
\end{lem}
\begin{proof}
Since $D$ has boundary of class $C^1$ (and hence $D$ is a $\H$-Caccioppoli set(\cite{fsc}, Definition 2.11)), so, Green's first identity is just the Divergence theorem (\cite{fsc}, Corollary 7.7) applied to the vector field $v\nabla_0 u$.

\end{proof}
\begin{lem} \label{x}
If $u$ is non constant on a connected domain $D$ then $\nabla_0 u$ is non zero on some open subset $U$ of $D$.
\end{lem}
\begin{proof}
Suppose $\nabla_0 u$ is zero everywhere. Consider the level sets $S_c=\{[z,t]\in \H_n:u([z,t])=c,\;\; c\in \R\}$. These would give a foliation of $D$ by hypersurfaces. Since $\nabla_0 u$ is zero at every point, the Euclidean gradient $\nabla u$ is contained in $\text{span} \{T\}$ and consequently tangent space at any point is in span$\{X_j,Y_j\}$. Therefore span$\{X_j,Y_j\}$ is a distribution on $E$ which has an integral surface through any point $p$ viz $S_{u(p)}$. Thus, by Frobenius theorem (\cite{tay}, Theorem 9.4) span$\{X_j,Y_j\}$ must be closed under bracket operation which is a contradction as $$[X_j, Y_j]=-4T\;\;\;\; \text{for}\;\; 1\leq j \leq n.$$\\ Hence the lemma.    
\end{proof}
\begin{thm}
Let $D$ be a domain with smooth boundary. Two solutions of the interior Neumann problem (\ref{7}) can differ only by a constant.
\end{thm}
\begin{proof}
The difference $u:=u_1-u_2$ of two solutions for the Neumann problem is a harmonic function continuous up to the boundary satisfying the homogeneous boundary condition $\partial^\perp u=0$ on $\partial D$. By Lemma \ref{y} applied to $D$, we derive
$$\int_D{|\nabla_0 u|^2\; d\eta}= -\int_{D}{u(L_0u) d\eta}+\int_{\partial D}{u\partial^\perp u d\sigma}=0.$$Thus $\nabla_0 u=0$ on $D$, so by Lemma \ref{x}, $u$ must be constant.
\end{proof}
\subsection{Surface Potentials}
From now onwards, $B$ will denote the Kor\'{a}nyi ball in $\H_n$, i.e, $B=\{\xi=[z,t]\in \H_n:\; N(\xi)<1\}.$
\begin{defn}
Given a function $\phi \in C(\partial B)$, for $\eta \in \H_n\setminus \partial B, $ the functions
\beg \label{15}
m(\eta):=\int_{\partial B}{\phi(\xi)\;g_{\eta}(\xi)\;d\sigma(\xi)},
\eeg
and
\beg \label{16}
v(\eta):=\int_{\partial B}{\phi(\xi)\;\partial^\perp  g_{\eta}\;d\sigma(\xi)},
\eeg
are called, respectively, single-layer and double-layer potential with density $\phi$. Both potentials are $L_0$-harmonic.
\end{defn}
In the following results, we determine the limiting values of single layer and double layer potentials. To that end, for $\eta\in \partial B$, we define $$v_{\pm}(\eta)=\lim_{h \rightarrow +0}v(\eta \pm h n_0(\eta)),$$ and,
$$\partial^\perp m_{\pm}(\eta)=\lim_{h \rightarrow +0}n_0(\eta). \nabla m(\eta \pm h n_0(\eta)).$$ The limit is to be understood in the sense of uniform convergence on compact subsets of $\partial B$.
\begin{lem}\label{**}
For continuous density $\phi$ on $\partial B$, 
\beg \label{*}
\int_{\partial B}{\partial^\perp g_\eta(\xi)\phi(\xi) d\sigma(\xi)},
\eeg exist for each $\eta \in \partial B.$
\end{lem}
\begin{proof}
From \cite{kori} $$\partial^\perp=\frac{1}{|z|}(\bar{A}E+A\bar{E}),\; \text{for all points where}\; |z|\neq 0,$$ where $E=\sum{z_j Z_j}$ and $A=|z|^2+it$.
For $\xi=[z,t]$ and $\eta=[z',t'] \in \H_n $ and $|z| \neq 0$ we have 
\ben
\frac{\partial}{\partial n_0} g_\eta(\xi)&=&-n \frac{a_0}{|z|}(N(\xi^{-1}\eta))^{(-2n-2)}[4|z|^6+2|z|^2|z'|^4+8|z|^2z^2\bar{z'}^2+4|z|^4|z'|^2 \\
&& -4z\bar{z'}|z|^2|z'|^2-12|z|^4z\bar{z'}-8iz^2\bar{z'}^2t-4iz\bar{z'}|z'|^2t-12i|z|^2z\bar{z'}t\\
&&+4|z|^2t^2-4|z|^2tt'],
\een and $$\lim_{|z| \rightarrow 0}\frac{\partial}{\partial n_0} g_\eta(\xi)=0,$$ so, $g_\eta(\xi) \in \emph{C}_*(B)$ and 
\beg \label{18}
\left|\partial^\perp g_\eta(\xi)\right| \leq C_1|a_0|.n(N(\xi^{-1}\eta))^{(-2n-1)},\;\;\;\eta \neq \xi,
\eeg for some constant $C_1 >0.$\\
For each $\eta \in \partial B$, let $B_\eta(R)=\{\xi \in \partial B: N(\xi \eta^{-1}) \leq R\}$, $R \in (0,1].$ $B_\eta(R)$ can be projected bijectively onto the tangent plane to $\partial B$ at the point $\eta$.\\ Since measure is translation invariant, it is enough to estimate the integral in (\ref{*}) with $g_e(\xi).$ For this we introduce polar coordinates. For the sake of brevity, we shall do it for $\H_1$ only but everything holds well in $\H_n$. In $\H_1$ the polar coordinates $r,\; \phi,\;\alpha$ of a point $\xi=[\rho e^{i\theta},t]$ as in \cite{kori} are given by
\ben
\rho&=& r\cos^{{1}/{2}}\alpha,\\
t&=&r^2 \sin \alpha,\\
\theta&=& \phi+\tan \alpha \log \frac{r}{a},
\een
with fixed $a$. We can write
\ben
\left|\int_{B_\epsilon(R)}{\partial^\perp g_e(\xi)\phi(\xi) d\sigma(\xi)}\right|&\leq& \frac{C_1 |a_0|\;n\left\|\phi\right\|_\infty}{\left\|\nabla N\right\|}\int_{B_\epsilon(R)}{|z|N(\xi)^{-2n-1}ds(\xi)}\\
&=& \frac{C_1 |a_0|\;n\left\|\phi\right\|_\infty}{\left\|\nabla N\right\|} \int_0^{2\pi}{d\phi}\int_{{-\pi}/{2}}^{{\pi}/{2}}{|\cos^{{1}/{2}}\alpha|d\alpha}\int_0^R{rdr},
\een
where $B_\epsilon(R)$ is gauge ball with center at origin and radius $R$. Hence integral exist on $B_\eta(R)$. Furthermore, we have
\ben
\left|\int_{\partial B\setminus{B_\eta(R)}}{\partial^\perp g_\eta(\xi)\phi(\xi) d\sigma(\xi)}\right|&=& \frac{C_1 |a_0|\;n\left\|\phi\right\|_\infty}{\left\|\nabla N\right\|}\int_{\partial B\setminus{B_\eta(R)}}{R^{-2n-1}d\sigma(\xi)}\\
&\leq& \frac{C_1 |a_0|\;n\left\|\phi\right\|_\infty}{\left\|\nabla N\right\|}|\partial B|.
\een
Hence, for all $\eta \in \partial B$, the integral in (\ref{*}) exist.

\end{proof}
\begin{thm} \label{a}
The double-layer potential $v$ with continuous density $\phi$ can be continuously extended from $B$ to $\bar{B}$ and from $\H_n \setminus \bar{B}$ to $\H_n \setminus B$ with limiting values 
\beg \label{17}
v_{\pm}(\eta)=\int_{\partial B}{\phi(\xi) \partial^\perp  g_{\eta}(\xi)\;d\sigma(\xi)} \pm \phi(\eta),\;\; \eta\in \partial B.
\eeg
\end{thm}
\begin{proof}
By Lemma \ref{**} the integral exists for $\eta \in \partial B$ and represents a continuous function on $\partial B$. In a sufficiently small neighborhood $U$ of $\partial B$, we can represent almost all $\eta \in U$ uniquely in the form $\eta=\zeta+h n_0(\zeta),$ where $\zeta \in \partial B$ and $h \in [-h_0,h_0]$ for some $h_0 >0.$ Now, we write the double-layer potential $v$ with density $\phi$ in the form
$$v(\eta)=\phi(\zeta)w(\eta)+u(\eta),\;\;\;\;\;\; \eta=\zeta+h n_0(\zeta) \in U\setminus \partial B,$$
where,
\beg \label{19}
w(\eta)=\int_{\partial B}{\partial^\perp  g_{\eta}(\xi)d\sigma(\xi)},
\eeg
and
\beg \label{20}
u(\eta)=\int_{\partial B}{\{\phi(\xi)-\phi(\zeta)\}\partial^\perp  g_{\eta}(\xi)d\sigma(\xi)}.
\eeg
For $\eta \in \partial B$ i.e, $h=0$, the integral in (\ref{20}) exists and represents a continuous function on $\partial B$. By using (\ref{***}), we have $w(\eta)=1$, for $\eta \in \partial B,$ therefore, to establish the theorem it sufficies to show that $$\lim_{h \rightarrow 0}u(\zeta+h n_0(\zeta))=u(\zeta),\;\;\;\; \zeta \in \partial B,$$ uniformly on compact subsets of $\partial B$.\\ Denoting $\partial B(\zeta;r)=\partial B \cap B[\zeta;r],$ where $B[\zeta;r]=\{\xi \in \H_n:N(\xi^{-1}\zeta)\leq r\}.$ Take $r< N(\eta^{-1}\zeta)=\alpha$(say) for $r$ sufficiently small, $\alpha-r$ is lower bound of $N(\xi^{-1}\eta).$ Consider
\beg
\int_{\partial B(\zeta;r)}{\left|\partial^\perp  g_{\eta}(\xi)\right|d\sigma(\xi)}&\leq& C_1 \int_{\partial B(\zeta;r)}{(N(\xi^{-1}\eta))^{(-2n-1)}d\sigma(\xi)},\eta \neq \xi \nonumber\\
&\leq& C_1 \int_{\partial B(\zeta;r)}{\frac{1}{(\alpha-r)^{(2n+1)}}d\sigma(\xi)}\nonumber\\
&\leq& C_1 \frac{1}{(\alpha-1)^{(2n+1)}}|\partial B(\zeta;r)| \label{21}\\
&=& C_2 \text{(say)}.\nonumber
\eeg
From the mean value theorem we obtain, 
\ben
\left|\partial^\perp  g_{\eta}(\xi)- \partial^\perp  g_{\zeta}(\xi)\right|&\leq& C_3 N(\eta \zeta^{-1}).(\nabla)_{\zeta}\left(\partial^\perp  g_{\eta}(\xi)\right)\\
&\leq& C_4 \frac{N(\eta \zeta^{-1})}{(N(\zeta \xi^{-1}))^{2n+1}}.
\een
Hence we can estimate 
\beg \label{22}
\int_{\partial B\setminus \partial B(\zeta;r)}{\left|\partial^\perp  g_{\eta}(\xi)\right|d\sigma(\xi)}&\leq& C_5 \frac{N(\eta\zeta^{-1})}{r^{2n+1}},
\eeg
for some constant $C_5>0.$ Now we can combine (\ref{21}) and (\ref{22}) to show that
$$|u(\eta)-u(\zeta)| \leq C \left\{\max_{B[\zeta;r]}|\phi(\xi)-\phi(\zeta)|+\frac{N(\eta \zeta^{-1})}{r^{2n+1}}\right\}$$ for some constant $C>0$ and all sufficiently small $r$.\\
Given $\epsilon>0$ we can choose $r>0$ such that $$\max_{\xi \in B[\zeta;r]}|\phi(\xi)-\phi(\zeta)| \leq \frac{\epsilon}{2C},$$ for all $\zeta \in \partial B,$ since $\phi$ is uniformly continuous on $\partial B$. Then taking, \\ $\delta < \frac{\epsilon r^{2n+1}}{2C}$, we see that $|u(\eta)-u(\zeta)|< \epsilon$ for all $N(\eta \zeta^{-1})< \delta$.
\end{proof}
\begin{thm} \label{b}
For a single-layer potential $m$ with continuous density $\phi$, we have
$$\partial^\perp  m_{\pm}(\eta)=\int_{\partial B}{\phi(\xi)\;\partial^\perp  g_{\eta}(\xi)\;d\sigma(\xi)} \pm \phi(\eta).\;\;\; \eta \in \partial B.$$
\end{thm}
\begin{proof}
let $v$ denote the double-layer potential with density $\phi$ and let $U$ be as in the proof of Theorem \ref{a}. Then for $\eta=\zeta+ h n_0(\zeta) \in U\setminus \partial B$, we can write
$$n_0(\zeta). \nabla (m(\eta))+v(\eta)=\int_{\partial B}{\{n_0(\zeta)+n_0(\eta)\}. \nabla_{\xi} (g_{\eta}(\xi))\;\phi(\xi)\;d\sigma(\xi)},$$ where we have made use of $(\nabla)_{\eta}(g_{\eta}(\xi))=(\nabla)_{\xi}( g_{\eta}(\xi))$.\\ Analogous to the single-layer potential in Theorem \ref{a}, the right hand side can be seen to be continuous in $U$. The proof can be now completed by applying Theorem \ref{a}.
\end{proof}
\begin{thm}\label{c}
A double-layer potential $v$ with continuous density $\phi$ satisfies
$$\lim_{h\rightarrow +0} n_0(\eta).\{\nabla v(\eta+h n_0(\eta))-\nabla v(\eta-h n_0(\eta))\}=0,$$ uniformly for all $\eta \in \partial B$.
\end{thm}
\begin{proof}
The proof is similar in structure to the proof of Theorem \ref{a}.
\end{proof}
\subsection{Neumann Boundary Value Problem: Existence}
We introduce two integral operators $K,K':C(\partial B)\rightarrow C(\partial B)$ by
$$(K \phi)(\eta):=\int_{\partial B} {\phi(\xi)\;(\partial^\perp  g_{\eta}(\xi))_\xi\;d\sigma(\xi)},\;\;\;\;\; \eta \in \partial B,$$ and
$$(K' \psi)(\eta):=\int_{\partial B} {\psi(\xi)\;(\partial^\perp  g_{\eta}(\xi))_\eta\;d\sigma(\xi)},\;\;\;\;\; \eta \in \partial B.$$
Using the estimate (\ref{18}), it can be easily shown that operators $K$ and $K'$ are compact. As seen by interchanging the order of integration, $K$ and $ K' $ are adjoint with respect to the dual system $\langle C(\partial B),C(\partial B)\rangle $ defined by $$ \langle \phi,\psi \rangle:=\int_{\partial B}{\phi \psi d\sigma},\;\;\;\; \phi, \psi \in C(\partial B). $$
\begin{thm} \label{d}
The null spaces of the operators $I+K$ and $I+K'$ have dimension one.
\end{thm}
\begin{proof}
Let $\phi$ be a solution of $\phi + K \phi=0$ and again define double layer potential $v$ as in (\ref{16}). Then by (\ref{17}), $v_+=K\phi+\phi=0$ on $\partial B$. $v(\eta)$ is a bounded harmonic function on $\H_n\setminus \bar{B},$ applying maximum principle \cite{bolaug} to $\H_n\setminus \bar{B},$ we get $v=0$ on $\H_n\setminus \bar{B}$. From Theorem \ref{c}, we see that $\partial^\perp v_-=0$ on $\partial B$ and from the uniqueness of the interior Neumann problem, it follows that $v$ is constant in $B$. From (\ref{17}) we deduce that $\phi$ is constant on $\partial B$. Therefore, Null space of $(I+K)$ $\subset$ span $\{c_0\}$ where $c_0$ is a constant and by using (\ref{***}), we get
$$\int_{\partial B}{\partial^\perp g_{\eta}(\xi)\;d\sigma(\xi)}=1,$$ so we have $c_0+Kc_0=0.$ Thus Null space of $(I+K)$ = span $\{c_0\}$. By First Fredholm theorem (\cite{kres}, Theorem 4.15), Null space of $(I+K')$ also has dimension one.  
\end{proof}
\begin{thm}
The single-layer potential $$m(\eta):=\int_{\partial B}{\psi(\xi)\;g_{\eta}(\xi)\;d\sigma(\xi)},\;\;\;\;\; \eta \in B$$ with continuous density $\psi$ is a solution of the interior Neumann problem (\ref{7}) provided that $\psi$ is a solution of the integral equation
$$\psi(\eta)+\int_{\partial B}{\psi(\xi)\;\partial^\perp g_{\eta}(\xi)\;d\sigma(\xi)}=g(\eta),\;\;\;\; \eta \in \partial B.$$
\end{thm}
\begin{proof}
This follows from Theorem \ref{b}.
\end{proof}
\begin{thm} \label{e}
The interior Neumann problem is solvable if and only if $$\int_{\partial B}{g\; d\sigma}=0,$$ is satisfied. 
\end{thm}
\begin{proof} \textbf{Necessary Part}\\
Using identity (\ref{***}) for a solution $u$ of (\ref{7}) and $v=1$, we have $$\int_{\partial B}{g\; d\sigma}=0.$$
\textbf{Sufficient Part}\\
Its sufficiency follows from the fact that by Theorem \ref{d} it coincides with the solvability condition of the Fredholm alternative for the inhomogeneous integral equation i.e,$\psi+K'\psi=g.$ \\ By Fredholm alternative for integral equations of the second kind stated as a corollary in (\cite{kres}, Corollary 4.18) the solution of the interior Neumann problem exist if $\int_{\partial B}{g\; d\sigma}=0.$ 
\end{proof}
\section{Neumann fuction for kor\'{a}nyi ball} 
In this section we construct an explicit representation formula for Neumann function for Laplacian over the Kor\'{a}nyi ball in $\H_n$ by means of the fundamental solution for Laplacian, Kelvin tranform and spherical harmonics. Then, making use of this Neumann function, the solution of the Neumann problem for the Poisson equation is given explicitly.\\ Let $B$ denotes the Kor\'{a}nyi ball in $\H_n$, i.e, $B=\{\xi=[z,t]\in \H_n: N(\xi)<1\}$.\\ For a function $f$ on $\H_n$, the Kelvin transform is defined as in \cite{kor82} by $$Kf= N^{-2n}foh,$$ where $h$ is the inversion defined as $$h([z,t])=\left[\frac{-z}{|z|^2-it},\frac{-t}{|z|^4+t^2}\right],$$ for $[z,t] \in\H_n\setminus\{e\}$.\\ For complex $\alpha, \beta$, the functions $C_m^{(\alpha, \beta)}$ $(m=0, 1, 2, \ldots)$ are defined by the generating function $${(1-z\bar{\varsigma})}^{-\alpha}{(1-z\varsigma)}^{-\beta}=\sum_{m=0}^\infty{z^mC_m^{(\alpha, \beta)}(\varsigma, \bar{\varsigma})}, \;z,\;\varsigma \in \C,\;|z|<|\varsigma|^{-1}.$$ It follows immediately that $$C_m^{(\alpha, \beta)}(\varsigma,\bar{\varsigma})=\sum_{p=0}^m{\frac{(\alpha)_{m-p}(\beta)_p}{(m-p)!p!}(\bar{\varsigma})^{m-p}\varsigma^p},\; \varsigma\in\C.$$ Special case $C_m^{(\alpha, \alpha)}(\varsigma,\bar{\varsigma})$ is denoted by Gegenbauer polynomial. In \cite{grko}, by using the function $C_m^{(\alpha, \beta)}$, the fundamental solution and its Kelvin transform are expressed as

\ben
g_{\eta^{-1}}(\xi^{-1})&=&(|z|^4+t^2)^{-n/2}\sum_{m=0}^\infty\sum_{k,l=0}^\infty\sum_{j=1}^{N_{k,l}}a_{m;k,l}\;i^{k-l}(|z|^2+it)^{-m-l}(|z|^2-it)^{-m-k}\\
&&\times C_m^{(\frac{n}{2}+l, \frac{n}{2}+k)}(t+i|z|^2) Y_{k,l;j}(z)C_m^{(\frac{n}{2}+l, \frac{n}{2}+k)}(t'+i|z'|^2)\overline{Y_{k,l;j}(z')},
\een

$$K(g_{\eta^{-1}}(\xi))=\sum_{m=0}^\infty\sum_{k,l=0}^\infty\sum_{j=1}^{N_{k,l}}{a_{m;k,l}C_m^{(\frac{n}{2}+l, \frac{n}{2}+k)}(t+i|z|^2)Y_{k,l;j}(z)C_m^{(\frac{n}{2}+l, \frac{n}{2}+k)}(t'+i|z'|^2)\overline{Y_{k,l;j}(z')}},$$ where $N_{k,l}$ is dimension of the space $H_{k,l}$. The space $H_{k,l}$ of complex (solid) spherical harmonics of bidegree $(k,l)$ on $\C^n$ consists of all polynomials $P$ in $z_1,\ldots z_n,\;\bar{z_1},\ldots \bar{z_n}$, homogeneous of degree $k$ in the $z_j's$ and homogeneous of degree $l$ in the $\bar{z_j}'s$  satisfying $$\sum_{j=1}^n{\frac{\partial^2}{\partial z_j \partial \bar{z_j}}P}=0.$$ For each $k,l$ the ${Y_{k,l;j}}'s$ form a basis for $H_{k,l}$, where ${Y_{k,l;j}}'s$ have the form 
$$Y_{k,l;j}(z)=\sum_{q=0}^r{c_q|z^*|^{2q}{z_1}^{k-q}{(\bar{z}_1)}^{l-q}},$$ $z=(z_1,z^*)\in\C^n,\; |z^*|^2=|z_2|^2+\ldots+|z_n|^2,\; \text{where},\;r=min(k,l)\;\text{and}\;c_0,\ldots,c_r$ are constants, $c_0=1$ and ${c_q}'s$ are determined by the relation $$(k-q)(l-q)c_q+(q+1)(n+q-1)c_{q+1}=0,\;0\leq q<r.$$ Now,
\ben
{\bar{Y}}_{k,l;j}&=&\frac{1}{2\pi}\int_0^{2\pi}{Y_{k,l;j}(ze^{i\theta})d\theta}\\
&=&\sum_{q=0}^k{c_q|z^*|^{2q}{|z_1|}^{2(k-q)}}\\
&=&Y_{k;j}(z),
\een and ${c_q}'s$ are determined by the relation $${(k-q)}^2c_q+(q+1)(n+q-1)c_{q+1}=0,\;0\leq q<k.$$ An easy computation yields the following expressions,
\ben
{\bar{g}}_{\eta^{-1}}(\xi^{-1})&=&(|z|^4+t^2)^{-n/2}\sum_{m=0}^\infty\sum_{k,l=0}^\infty\sum_{j=1}^{N_k}a_{m;k}(|z|^4+t^2)^{-m-k}C_m^{(\frac{n}{2}+k, \frac{n}{2}+k)}(t+i|z|^2)\\
&&\times Y_{k;j}(z)C_m^{(\frac{n}{2}+k, \frac{n}{2}+k)}(t'+i|z'|^2)\overline{Y_{k;j}(z')},
\een and,$$K({\bar{g}}_{\eta^{-1}}(\xi))=\sum_{m=0}^\infty\sum_{k,l=0}^\infty\sum_{j=1}^{N_k}{a_{m;k}C_m^{(\frac{n}{2}+k, \frac{n}{2}+k)}(t+i|z|^2)Y_{k;j}(z)C_m^{(\frac{n}{2}+k, \frac{n}{2}+k)}(t'+i|z'|^2)\overline{Y_{k;j}(z')}}.$$ The homogeneous Neumann problem is
\begin{eqnarray} 
L_0 N_B(\eta, \xi)&=&\delta_\eta \;\text{in}\; B, \label{11}\\
\partial^\perp N_B(\eta, \xi)&=&0 \;\text{on}\; \partial B \label{12}.
\end{eqnarray} Here $L_0$ denotes the Laplacian and $\partial^\perp$ denotes the outward pointing horizontal normal unit vector on the boundary of the Kor\'{a}nyi ball and it is explicitly given in \cite{kori} as $$\partial^\perp=\frac{1}{|z|}(\bar{A}E+A\bar{E}),\; \text{for all points where}\; |z|\neq 0,$$ where $E=\sum{z_j Z_j}$ and $A=|z|^2+it$. Now we try to look for a solution of problem $(\ref{11})$ and $(\ref{12})$ together in the following form $$N_B(\eta, \xi)={\bar{g}}_{\eta^{-1}}(\xi^{-1})+K({\bar{g}}_{\eta^{-1}}(\xi))+h_m(\eta,\xi),$$ where ${\bar{g}}_{\eta^{-1}}(\xi^{-1})$ is the fundamental solution for the Laplacian, namely\\ $L_0 {\bar{g}}_{\eta^{-1}}(\xi^{-1})=\delta_\eta$ and $K({\bar{g}}_{\eta^{-1}}(\xi))$ is harmonic in $B$. So, the function $h_m(\eta,\xi)$ is a circular function that satisfies
\begin{eqnarray} 
L_0 h_m(\eta, \xi)&=&0,\;\xi\in B, \label{13}\\
\partial^\perp h_m(\eta, \xi)&=&-\frac{\partial}{\partial n_0}\left({\bar{g}}_{\eta^{-1}}(\xi^{-1})+K({\bar{g}}_{\eta^{-1}}(\xi))\right), \;\xi\in\partial B. \label{14}
\end{eqnarray} Suppose that there exists a solution for the problem $(\ref{13})$ and $(\ref{14})$ together  in the following form i.e, in the form of Heisenberg harmonics
\begin{equation*}
h_m(\eta,\xi)=\sum_{\substack {k\\
 m-2k\geq 0\\
 and\; even}}\sum_{j=1}^{N_k}{b_{m;k}(z',t')C_{\frac{1}{2}(m-2k)}^{(\frac{n}{2}+k,\frac{n}{2}+k)}(t+i|z|^2)Y_{k;j}(z)},
\end{equation*}
 where $b_{m;k}(z',t')$ are functions of $\eta=(z',t')$ to be determined.\\ Now, we claim that $N_B(\eta, \xi)={\bar{g}}_{\eta^{-1}}(\xi^{-1})+K({\bar{g}}_{\eta^{-1}}(\xi))+h_m(\eta,\xi)$ works as Neumann function when applied to circular functions.\\ Since all the three series of $\bar{g}_{\eta^{-1}}(\xi^{-1}),K({\bar{g}}_{\eta^{-1}}(\xi))$ and $h_m(\eta,\xi)$ are absolutely and uniformly convergent for $\xi \neq \eta$ and $K({\bar{g}}_{\eta^{-1}}(\xi))$ and $h_m(\eta,\xi)$ are harmonic. Therefore, $$L_0 N_B(\eta, \xi)=L_0{\bar{g}}_{\eta^{-1}}(\xi^{-1})=\delta_\eta.$$ We can easily calculate that $E(Y_{k;j}(z))=kY_{k;j}(z)$ and $\bar{E}(Y_{k;j}(z))=kY_{k;j}(z)$, so $$\partial^\perp(Y_{k;j}(z))=2|z|kY_{k;j}(z).$$ Computing the normal derivative of the fundamental solution, we have
\ben
\partial^\perp{\bar{g}}_{\eta^{-1}}(\xi^{-1})&=&(|z|^4+t^2)^{-n/2}\sum_{m=0}^\infty\sum_{k,l=0}^\infty\sum_{j=1}^{N_k}a_{m;k}(|z|^4+t^2)^{-m-k}C_m^{(\frac{n}{2}+k, \frac{n}{2}+k)}(t+i|z|^2)\\
&&\times C_m^{(\frac{n}{2}+k, \frac{n}{2}+k)}(t'+i|z'|^2)\overline{Y_{k;j}(z')}\left(\frac{\partial}{\partial n_0}(Y_{k;j}(z))+2|z|Y_{k;j}(z)(-m-2k-n)\right)\\
&=& (|z|^4+t^2)^{-n/2}\sum_{m=0}^\infty\sum_{k,l=0}^\infty\sum_{j=1}^{N_k}a_{m;k}(|z|^4+t^2)^{-m-k}C_m^{(\frac{n}{2}+k, \frac{n}{2}+k)}(t+i|z|^2)\\
&&\times C_m^{(\frac{n}{2}+k, \frac{n}{2}+k)}(t'+i|z'|^2)\overline{Y_{k;j}(z')}Y_{k;j}(z)2|z|(-m-k-n).
\een Similarly, we have 
\ben
\partial^\perp K({\bar{g}}_{\eta^{-1}}(\xi))&=&\sum_{m=0}^\infty\sum_{k,l=0}^\infty\sum_{j=1}^{N_k}a_{m;k}C_m^{(\frac{n}{2}+k, \frac{n}{2}+k)}(t+i|z|^2) C_m^{(\frac{n}{2}+k, \frac{n}{2}+k)}(t'+i|z'|^2)\\
&&\times \overline{Y_{k;j}(z')}Y_{k;j}(z)2|z|(m+k).
\een
 So, on the boundary of Kor\'{a}nyi ball i.e, at $N(\xi)=1$, we get
\ben
\partial^\perp({\bar{g}}_{\eta^{-1}}(\xi^{-1})+K({\bar{g}}_{\eta^{-1}}(\xi)))&=& -2n|z|\sum_{m=0}^\infty\sum_{k,l=0}^\infty\sum_{j=1}^{N_k}a_{m;k}C_m^{(\frac{n}{2}+k, \frac{n}{2}+k)}(t+i|z|^2)\\
&&\times C_m^{(\frac{n}{2}+k, \frac{n}{2}+k)}(t'+i|z'|^2)\overline{Y_{k;j}(z')}Y_{k;j}(z).
\een Also, at $N(\xi)=1$, we have
\begin{equation*}
\partial^\perp h_m(\eta,\xi)=\sum_{\substack{k\\
m-2k\geq 0\\
 and\; even}}\sum_{j=1}^{N_k}{b_{m;k}(z',t')C_{\frac{1}{2}(m-2k)}^{(\frac{n}{2}+k,\frac{n}{2}+k)}(t+i|z|^2)2|z|Y_{k;j}(z)\frac{m}{2}}.
\end{equation*} 
Therefore, to satisfy $(\ref{14})$, choose $b_{m;k}(z',t')$ such that $$b_{m;k}(z',t')=\frac{2n}{m}a_{m;k}C_m^{(\frac{n}{2}+k, \frac{n}{2}+k)}(t'+i|z'|^2)\overline{Y_{k;j}(z')},\;m \in \N.$$ Notice that when $m=0$ so $k=0$ and $C_0^{(\frac{n}{2}, \frac{n}{2})}(t'+i|z'|^2)=1,\;\bar{Y}_{0;j}(z')=1$ thus we get
$$h_m(\eta,\xi)=\sum_{k=1}^\infty\sum_{m=1}^\infty\sum_{j=1}^{N_k}{\frac{2n}{m}a_{m;k}C_{\frac{1}{2}(m-2k)}^{(\frac{n}{2}+k,\frac{n}{2}+k)}(t+i|z|^2)Y_{k;j}(z)C_m^{(\frac{n}{2}+k, \frac{n}{2}+k)}(t'+i|z'|^2)\overline{Y_{k;j}(z')}+b_0},$$ where $b_0$ is a constant.\\ Since these variations over $m,\;k$ are infinite, so we have
$$\partial^\perp h_m(\eta, \xi)=-\partial^\perp \left({\bar{g}}_{\eta^{-1}}(\xi^{-1})+K({\bar{g}}_{\eta^{-1}}(\xi))\right), \;\xi\in\partial B.$$
Hence, $N_B(\eta, \xi)={\bar{g}}_{\eta^{-1}}(\xi^{-1})+K({\bar{g}}_{\eta^{-1}}(\xi))+h_m(\eta,\xi)$ works as Neumann function for $B$ when applied to circular functions.\\ The inhomogeneous circular Neumann boundary value problem for $B$ can be solved by using Neumann function $ N_B(\eta, \xi)$ which is proved as follows.
\begin{thm}
The inhomogeneous Neumann boundary problem
\begin{eqnarray} \label{25}
L_0 u&=&f \;\text{in}\; B,\nonumber \\
\partial^\perp u&=&g \;\text{on}\; \partial B,\; u \in \emph{C}_*(B),
\end{eqnarray}  where $f$ and $g$ are continuous circular functions, is solvable if and only if
\begin{eqnarray*} 
\int_B{f(\xi)dv(\xi)}=\int_{\partial B}{g(\xi)d\sigma(\xi)},
\end{eqnarray*} and the solution is given by the representation formula
\begin{eqnarray*} 
u(\eta)=\int_B{N_B(\eta,\xi)f(\xi)dv(\xi)}-\int_{\partial B}{N_B(\eta,\xi)g(\xi)d\sigma(\xi)}.
\end{eqnarray*} 
\end{thm}
\begin{proof}
\textbf{Necessary Part}\\ Using identity (\ref{***}) for a solution $u$ of inhomogeneous Neumann problem (\ref{25}) and $v=1$, we get the required solvability condition.\\
\textbf{Sufficient Part}\\
Firstly, we consider the inhomogeneous boundary value problem
\begin{eqnarray}\label{26}
L_0u_1&=&f \;\text{in}\; B,\nonumber \\
u&=&0 \;\text{on}\; \partial B,
\end{eqnarray}
where $f$ is continuous function.\\ This problem is solvable for every continuous function $f$ on the Kor\'{a}nyi ball in $\H_n.$ Hence, $\partial^\perp u_1$ has some value on $\partial B$.  \\ Next, we consider the homogeneous Neumann problem
\begin{eqnarray} \label{27}
L_0 u_2&=&0 \;\text{in}\; B,\nonumber \\
\partial ^ \perp u_2&=&g' \;\text{on}\; \partial B,\; u_2 \in \emph{C}_*(B)
\end{eqnarray}
where $g'=g-\partial^\perp u_1$ and $g$ is a continuous function.\\ By Theorem \ref{e}, interior Neumann problem (\ref{27}) is solvable if $\int_{\partial B}{g'}=0,$ $i.e,\; \int_{\partial B}{g}=\int_{\partial B}{\partial^\perp u_1}=\int_B{f}$ (by using (\ref{***})).\\Let $u=u_1+u_2.$\\ Since $L_0u=f$ and $\partial^\perp u=\partial^\perp u_1+g-\partial^\perp u_1=g$, therefore, $u$ is a solution of inhomogeneous Neumann boundary problem (\ref{25}) if 
$$\int_B{f(\xi)dv(\xi)}=\int_{\partial B}{g(\xi)d\sigma(\xi)}.$$ 
\textbf{Representation formula}\\
Take $u$ is a solution of (\ref{25}) and $v=N_B(\eta,\xi)$ in the identity (\ref{***}), we get
$$u(\eta)=\int_B{N_B(\eta,\xi)f(\xi)dv(\xi)}-\int_{\partial B}{N_B(\eta,\xi)g(\xi)d\sigma(\xi)},$$ where $f$ and $g$ are continuous circular functions.\\
This is the required representation formula for a solution of interior Neumann problem on the Kor\'{a}nyi ball in the Heisenberg group and holds for circular data only.
\end{proof}
\section*{acknowledgements}
 The first author is supported by the Senior Research Fellowship of Council of Scientific and Industrial Research, India (Grant no. 09/045(1152)/2012-EMR-I) and the second author is supported by R \& D grant from University of Delhi, Delhi, India.


\begin{thebibliography}{5}
\bibitem{bolaug} A. Bonfiglioli, E. Lanconelli, F. Uguzzoni, \textit{Stratified Lie Groups and Potential Theory for their sub-Laplacians}, Springer Monograph in Mathematics, Springer-Verlag Berlin-Heidelberg 2007.
\bibitem{copa} E. Constantin and N. H. Pavel, \textit{Green function of the Laplacian for the Neumann problem in $\R_+^n$}, Libertas Math. 30 (2010), pp. 57-69.
\bibitem{evan} Lawrence C. Evans, \textit{Partial differential equations}, AMS, Providence, Rhode Island, (1998).
\bibitem{fol} G. B. Folland, \textit{A Fundamental Solution for a subelliptic operator}, Bull. Am. Math. Soc.79 (1973), pp. 373-376.
\bibitem{fsc} B. Franchi, R. Serapioni and F. S. Cassano, \textit{Rectifiability and perimeter in the Heisenberg group}, Math. Ann. 321 (2001), pp. 479-531.
\bibitem{gav} B. Gaveau, \textit{Principe de moindre action, propagation da la chaleur et estim\'{e}es sous-elliptiques sur certaiins groupes nilpotents}, Acta Math, 139 (1977), pp. 95-153.
\bibitem{gav2} B. Gaveau, \textit{Syst$\grave{e}$mes dynamiques associ$\acute{e}$s $\grave{a}$ certains op$\acute{e}$rateurs hypoelliptiques}, Bull. Sc. Math., Paris, $2^c$ s$\acute{e}$rie, T. 102, 1978, pp. 203-229.
\bibitem{grko} P. C. Greiner and T. H. Koornwinder, \textit{Variations on the Heisenberg spherical harmonics}, Report ZW 186/83 Mathematisch Centrum, Amsterdam, 1983.
\bibitem{jeri} David S. Jerison, \textit{The Dirichlet problem for the Kohn Laplacian on the Heisenberg group}, I, J. Func. Anal. 43 (1981), pp. 97-142.
\bibitem{kor82} A. Kor\'{a}nyi, \textit{Kelvin transforms and harmonic polynomials on the Heisenberg group}, J. Funct. Anal. 49 (1982), pp. 177-185.
\bibitem{kor83} A. Kor\'{a}nyi, \textit{Geometric aspects of analysis on the Heisenberg group}, in ``Topics in modern harmonic analysis", Instituto nazionale di Alta mathematica, Roma, 1983, pp. 209-258.
\bibitem{kori} A. Kor\'{a}nyi and H. M. Riemann, \textit{Horizontal normal vectors and conformal capacity of spehrical rings in the Heisenberg group}, Bull. Sci. Math. Ser. 2 111 (1987), pp. 3-21.
\bibitem{kor6}  A. Kor\'{a}nyi, \textit{Poisson formulas for circular functions and some groups of type H}, Sci. China Ser. A: Math. 49 (2006), pp. 1683-1695.
\bibitem{kres} R. Kress, \textit{Linear Integral Equations}, Third Edition, Applied Mathematical Sciences, Springer, New York, 2014.
\bibitem{naya} B. M. Nayar, \textit{Neumann function for the sphere, I,II}, Indian J. Pure Appl. Math. 12, no. 10, 1981, 1266-1282, 1283-1292.
\bibitem{EDR} Earl D. Rainville, \textit{Special functions}, The MacMillan Company, New York, 1960.
\bibitem{stein} E. M. Stein, \textit{Harmonic analysis: Real-variable methods, orthogonality and oscillatory integrals}, Princeton Univ. Press, Princeton, New Jersey (1993).
\bibitem{tay} M. E. Taylor, \textit{Partial diffrential equations I}, Applied Math. Sci., Second Edition, Springer, New York (2011).
\bibitem{xu} Xu Zhenyuan, \textit{On boundary value problem of Neumann type for hypercomplex function with values in a Clifford algebra}, Circolo Mathematico di Palermo, Serie II, Supplemento No. 22 (1990), pp. 213-226.
\end{thebibliography}
\end{document}